\documentclass[12pt]{article}
\usepackage{amssymb}
\usepackage{graphicx}
\usepackage{amsmath}
\usepackage{harvard}

\newtheorem{theorem}{Theorem}

\newtheorem{condition}[theorem]{Condition}

\newtheorem{corollary}[theorem]{Corollary}

\newtheorem{example}[theorem]{Example}

\newtheorem{lemma}[theorem]{Lemma}
\newtheorem{notation}[theorem]{Notation}

\newtheorem{remark}[theorem]{Remark}

\newenvironment{proof}[1][Proof]{\textbf{#1.} }{\ \rule{0.5em}{0.5em}}
\input{tcilatex}
\renewcommand{\baselinestretch}{1}
\begin{document}

\title{{\vspace{-0.8in} \textbf{Perturbed Linear-Quadratic Control Problems
and Their Probabilistic Representations }}}
\author{Coskun Cetin\thanks{%
CSU, Department of Mathematics and Statistics, 6000 J St., Sacramento, CA
95819. Ph: (916) 278-6221. Fax: (916) 278-5586. Email: cetin@csus.edu}}
\date{December 2012}
\maketitle

\begin{abstract}
We consider some certain nonlinear perturbations of the stochastic
linear-quadratic optimization problems and study the connections between
their solutions and the corresponding Markovian backward stochastic
diferential equations (BSDEs). Using the methods of stochastic control,
nonlinear partial differential equations (PDEs) and BSDEs, we identify
conditions for the solvability of the problem and obtain some regularity
properties of the solutions.
\end{abstract}

\begin{quotation}
 \medskip

 \bigskip
\end{quotation}

\section{1. Introduction}

Stochastic optimal control problems and their connections with nonlinear
PDEs and backward stochastic differential equations (BSDEs) have been
subject of extensive research thanks to their wide range of potential
applications in engineering, financial economics and related areas. However,
explicit solutions to such problems can be obtained only in a few special
cases where the state and control variables usually appear linearly in the
state dynamics, for example, the linear-quadratic regulator (LQR) problems.
Both theoretical and computational issues arise when some nonlinear terms
are added to the system, sometimes representing the effect of a sudden
outside force (perturbation) to the variable of interest. In this paper, we
consider nonlinear perturbations only in the drift term of a state variable,
study the properties of the solution to the corresponding control problem
and then compare it with the standard (unperturbed) LQR problems. Our
approach involves using the connections between some quasilinear PDEs and a
class of forward-backward stochastic differential equations (FBSDEs) of the
following Markovian form 
\begin{eqnarray}
dX(t) &=&\mu (t,X(t))dt+\sigma (t,X(t))dW(t),\text{ }0\leq t\leq T  \notag \\
dY(t) &=&-F(t,X(t),Y(t),Z(t))dt+Z(t)dW(t),\text{ }0\leq t\leq T  \label{p0}
\\
X(0) &=&x;\text{ \ }Y(T)=g(X(T))  \notag
\end{eqnarray}%
where $\mu $ and $\sigma $ are the \textit{drift} and \textit{diffusion}
terms, respectively, of the forward process $X$; $F$ is the \textit{driver}
term of the backward process $Y$, and $Y(T)=g(X(T))$\ is the terminal
condition. We refer the reader to the books by Yong and Zhou (1999), Ma and
Yong (1999) and the survey paper by El Karoui et. al. (1997) for the general
theory and applications of backward stochastic differential equations
(BSDEs). The existence-uniqueness results for nonlinear BSDEs were provided
by Pardoux and Peng (1990), Mao (1995), Lepeltier and San Martin (1997,
1998), Kobylanski (2000), Briand and Hu (2006, 2008), and Cetin (2012),
among others.\ 

The connections between decoupled FBSDEs and quasilinear PDE's were first
stated by Pardoux and Peng (1992), and Peng (1992) by generalising
Feynman-Kac representation of PDE's. A version of their results that is
relevant to the system (\ref{p0}) is given below:

\begin{theorem}[Pardoux and Peng, 1992]
\label{t1}Consider the following parabolic PDE: 
\begin{eqnarray}
v_{{\large t}}+\mu v_{{\large x}}+F(t,x,v,\sigma v_{{\large x}})+\frac{1}{2}%
\sigma ^{2}v_{{\large xx}} &=&0  \label{k2.1} \\
v(T,x) &=&g(x),  \notag
\end{eqnarray}%
together with the decoupled system of FBSDEs (\ref{p0}). If the PDE (\ref%
{k2.1}) has a (classical) solution\ $v$, then the pair $(Y,Z)$\ with $%
Y^{s,x}(t)=v(t,X^{{\large s,x}}(t)),$ and\ $Z^{{\large s,x}}(t)=\sigma (t,X^{%
{\large s,x}}(t))v_{{\large x}}(t,X^{{\large s,x}}(t))$ solve the BSDE in (%
\ref{p0}). Conversely, if the system (\ref{p0}) has a unique (adapted)
solution, then $v(t,x)\triangleq Y^{t,x}(t)$\ is a viscosity solution to the
PDE (\ref{k2.1}). Moreover, this solution is unique if the coefficients
involved are uniformly Lipshitz.
\end{theorem}

We provide the basic definitions and the notations of the paper below.
\medskip

\subsection{1.1\quad Definitions and Notations}

We consider the one-dimensional Euclidean space $%
\mathbb{R}
$, fixed time-horizon $[0,T]$ and a probability space $(\Omega ,\digamma ,P) 
$ where $\digamma =\{\digamma _{t}:0\leq t\leq T\}$ is the complete $\sigma
- $algebra generated by a Brownian motion process $W$. We define the
following function spaces:

\begin{itemize}
\item $C^{p,q}([0,T]\times 
\mathbb{R}
)$: The space of all real-valued measurable functions $f$ $:[0,T]\times 
\mathbb{R}
$ such that $f(t,x)$ is $p$ (respectively, $q$) times continuosly
differentiable with respect to $t$ (respectively, $x$) where $p,q$ are
non-negative integers.

\item $L_{\digamma _{T}}^{p}(\Omega )$: The space of $\digamma _{T}$%
-measurable random variables $H$\ such that $E[\left\vert H\right\vert
^{p}]<\infty $.

\item $L_{\digamma _{T}}^{\infty }(\Omega )$: The space of $\digamma _{T}$%
-measurable essentially bounded\ random variables.

\item $L_{\digamma }^{p}([0,T])$: The space of $\digamma $-adapted processes 
$f$ such that $E[\int\limits_{0}^{T}\left\vert f(t)\right\vert
^{p}dt]<\infty $.

\item $L_{\digamma }^{\infty }([0,T])$: The space of $\digamma $-adapted
essentially bounded processes.

\item $S_{\digamma }^{p}(C[0,T])$: The space of $\digamma $-adapted
continuous processes such that $E[\sup\limits_{0\leq t\leq T}\left\vert
f(t)\right\vert ^{p}dt]<\infty .$
\end{itemize}

For a deterministic function $h(t,x):[0,T]\times 
\mathbb{R}
\rightarrow 
\mathbb{R}
$, the subscript notation denotes partial derivatives: $h_{t}(t,x)=\frac{%
\partial h}{\partial t}(t,x)$, $h_{x}(t,x)=\frac{\partial h}{\partial x}%
(t,x) $ and $h_{xx}(t,x)=\frac{\partial ^{2}h}{\partial x^{2}}(t,x)$. In
particular, for time dependent functions or ODE's, dot\ ($^{\cdot }$)
designates the derivative with respect to time parameter $t$. The notation $%
E_{t}[.]$ will denote the conditional expectation $E[.|\digamma _{t}]$. When
the initial value of a process $X$ is given at time $t$, then $E_{t,x}[.]\ $%
refers to $E[.]\ $with $X_{t}=x$. We also use the following facts and
notation related to the stochastic control theory. For the details and the
proofs of these arguments, one can refer to the books by Fleming and Rishel
(1975), Fleming and Soner (2006), or Yong and Zhou (1999).

Let $U\subset 
\mathbb{R}
$ and $x_{0}\in 
\mathbb{R}
.$\ Then for $U$-valued, $\digamma _{t}$-adapted control processes $u_{t}$,
consider the following control dependent SDE:%
\begin{eqnarray}
dX_{t} &=&a(t,X_{t},u_{t})dt+\sigma (t,X_{t})dW_{t}  \label{p1} \\
X_{0} &=&x_{0}\   \notag
\end{eqnarray}%
Consider also the cost functional 
\begin{equation}
J^{u}(s,x)=E_{s,x}[\int\limits_{s}^{T}f(t,X_{t},u_{t})dt+g(X_{T})]
\label{p2}
\end{equation}%
where the \textit{running cost} $F$, the \textit{terminal cost} $g$ and the
control $u$ belong to some appropriate $L_{\digamma }^{p}$ or $L_{\digamma
_{T}}^{p}$ space. We then define the \textit{value function} as 
\begin{equation}
V(s,x)=\inf\limits_{u\in \mathcal{U}^{s,x}}J^{u}(s,x)  \label{p3}
\end{equation}%
where $\mathcal{U}^{s,x}$\ is the set of all \textit{admissible controls,}
which consists of all $\{\digamma _{t},0\leq s\leq t\leq T\}$ adapted
processes $u=\{u(t),s\leq t\leq T\}$ with values in $%
\mathbb{R}
$ such that $E[\int\limits_{0}^{T}\left\vert u(t)\right\vert ^{2}dt]<\infty $
and the state process $X^{s,x;u}=X$\ in (\ref{p1}) has a unique strong
solution in $L_{\digamma }^{2}$. When there is no ambiguity, the notations $%
X^{s,x;u},$ $\mathcal{U}([0,T],%
\mathbb{R}
)$ will be abbreviated as $X$ and $\mathcal{U},$ respectively. We use
subscript notation $X_{t}$ and $u_{t}$\ for the state and control processes
when the context is clear. If a pair $(X^{\ast },u^{\ast })\ $is optimal for
the problem (\ref{p1})-(\ref{p3}) and the value function satisfies $V\in
C^{1,2}([0,T]\times 
\mathbb{R}
)$, then by dynamic programming principle (DPP) and standard verification
theorems, $V$ solves the following second-order nonlinear PDE: 
\begin{equation}
0=\inf\limits_{u}\{f(t,x,u)+(\mathsf{L}^{u}v)(t,x)\}  \label{p8}
\end{equation}%
with terminal condition $v(T,x)=g(x),$ where $\mathsf{L}$\textsf{\ }is the
backward evolution operator 
\begin{equation}
\mathsf{L}^{u}v(s,x)=v_{s}(s,x)+a(s,x,u)v_{x}(s,x)+\frac{1}{2}\sigma
^{2}v_{xx}(s,x).  \label{L}
\end{equation}%
The equation (\ref{p8}) is called the \textit{Hamilton-Jacobi-Bellman (HJB)
equation} (or \textit{HJB PDE}). The PDE is called \textit{uniformly
parabolic} if $\exists $ $c>0$ such that $\left\vert \sigma (t,x)\right\vert
\geq c$ for all $(t,x)\in \lbrack 0,T]\times 
\mathbb{R}
$. Such PDEs are known to have unique classical solutions under some
regularity and growth conditions, see for example, Fleming and Soner (2006,
IV.4) or Yong and Zhou (1999).

\begin{notation}
We write $u\in \mathcal{U}^{s,x}$ or $u\in \mathcal{U}$ to refer to an
admissible control system $(\Omega ,\digamma ,P,X(.),u(.))$\ when the
context is clear.
\end{notation}

The next section briefly describes how such a system of FBSDEs can be used
to study the properties of the solutions to some certain quasilinear HJB
PDEs corresponding to the stochastic optimal control problems of the form (%
\ref{p1})-(\ref{L}) where only the drift term of the state process is
control-dependent. Such an example is the linear-quadratic regular (LQR)
problems where the diffusion term is control-free, and the value function
has an explicit (quadratic) form which can be solved analytically or
numerically. However, when the drif term of the state equation is not
linear, an explicit form of the value function may not be available since
the corresponding HJB equation doesn't have an analytic solution, in
general. This occurs in the perturbed LQR problems with the drift
coefficient having extra nonlinear terms. A class of such nonlinear
perturbation were studied by Tsai (1978) without a terminal cost term. See
also Nishikawa et al. (1976). For more general cases, some generalized (e.g. 
\textit{viscosity}) solutions should be considered. Even if a smooth
solution exists, there are some other issues to consider: Uniqueness,
regularity properties and the numerical computation of the solutions.

\section{2. The PDE and FBSDE Representations}

In this section, we first assume that the stochastic control problem (\ref%
{p1})-(\ref{p3})\ is solvable with the value function $V(t,x)\in
C^{1,2}([0,T]\times 
\mathbb{R}
)$, and state a general (weak) representation formula for the corresponding
FBSDE system, in the spirit of Theorem \ref{t1}.\ However, we will consider
strong solutions in the rest of the paper.

\begin{lemma}
\label{l1} Let $V(t,x)\in C^{1,2}([0,T]\times 
\mathbb{R}
)$ be a solution to the optimization problem (\ref{p1})-(\ref{p3}) and $%
u^{\ast }=\arg \min\limits_{u}\{f(t,x,u)+a(t,x,u)V_{x}(t,x)\}$ which depends
on $t$ and $x$ deterministically through $u^{\ast }(t,x)=\pi
(t,x,V_{x}(t,x)) $, representing an optimal Markovian control rule $u^{\ast
}(t,X_{t})$. Moreover, assume that for some $p\geq 2$ and for all $s\in
\lbrack 0,T)$, the stochastic integral equation $X_{t}^{s,x}=x+\int%
\limits_{s}^{t}\sigma (r,X_{r})dW_{r}$\ has a weak solution ($\hat{X},\hat{W}%
,\hat{\digamma}$) in $L_{\hat{\digamma}}^{p}([s,T])$ and $f(t,\hat{X},\pi (t,%
\hat{X},V_{x}(t,\hat{X})\in L_{\hat{\digamma}}^{1}([s,T])$. Then

(i) $V$ solves the PDE%
\begin{eqnarray*}
v_{t}+\frac{1}{2}\sigma ^{2}v_{xx}(t,x)+\hat{F}(t,x,\sigma v_{x}(t,x)) &=&0
\\
v(T,x) &=&g(x),
\end{eqnarray*}%
where $\hat{F}(t,x,z)=a(t,x,\pi (t,x,\sigma ^{-1}(t,x)z)+f(t,x,\pi
(t,x,\sigma ^{-1}(t,x)z).$

(ii) A solution to the system 
\begin{equation}
Y_{t}^{s,x}=g(\hat{X}_{T})+\int\limits_{t}^{T}\hat{F}(r,\hat{X}%
_{r},Z_{r})dr-\int\limits_{t}^{T}Z_{r}d\hat{W}_{r}  \label{bsd}
\end{equation}%
is given by $Y_{t}^{s,x}=V(t,\hat{X}_{t}^{s,x})$, $Z_{t}^{s,x}=\sigma
V_{x}(t,\hat{X}_{t}^{s,x})$.
\end{lemma}

\begin{proof}
The part (i) follows from the stochastic optimal control and DPP arguments
in the previous section. For part (ii), let the operator $\mathsf{L}^{u}$\
be as in (\ref{L})\ corresponding to the SDE (\ref{p1}).\ Then applying
Ito's rule to $Y(t)=V(t,\hat{X}(t))$\ and using (\ref{p8}) with $u^{\ast
}=u^{\ast }(t,\hat{X})=\pi (t,\hat{X},\sigma V_{x}(t,\hat{X})$, we get: 
\begin{eqnarray*}
dY &=&\{\mathsf{L}^{u^{\ast }}V(t,\hat{X})-a(t,\hat{X},\pi (t,\hat{X}%
,V_{x}(t,\hat{X}))\}dt+\sigma (t,\hat{X})V_{x}(t,\hat{X})d\hat{W} \\
&=&-\{f(t,\hat{X},\pi (t,\hat{X},V_{x}(t,\hat{X}))+a(t,\hat{X},\pi (t,\hat{X}%
,V_{x}(t,\hat{X}))\}dt+Zd\hat{W} \\
&=&-\hat{F}(t,\hat{X},Z)dt+Zd\hat{W},
\end{eqnarray*}%
which also corresponds to the BSDE representation in (\ref{p0}) and (\ref%
{k2.1}), with $\mu =0$ .\footnote{%
Note that the Brownian motion process W in equation (\ref{bsd}) is not
necessarily the original one due to the weak representation approach.}
\end{proof}

Under some standard regularity and growth conditions (Lipshitz parameters,
linear growth etc.), the HJB PDEs (and the corresponding BSDEs) have unique
solutions. However, the equations that we consider are quite general and
there is no guarantee that a (classical) solution exists or if it exists
whether the solution is unique (in a suitable space). In the next
subsection, our particular interest will be on a more specific modeling of
the state variable $X$ and the cost functional $f$, which has some
interesting applications in economics (controlling macroeconomic variables)
and in engineering (target tracking). We keep the model coefficients as
general as possible in this section to motivate the choice of the particular
functions used in the next section.

\subsection{2.1\quad The Additively Separable Drift and Cost Functions}

We consider the following controlled state process $X=X^{u},$ for $0<t\leq
T: $ 
\begin{eqnarray}
dX_{t} &=&(\mu (t,X_{t})+B(t,X_{t})u_{t})dt+\sigma (t,X_{t})dW(t),  \notag \\
X_{0} &=&x_{0}.  \label{6}
\end{eqnarray}%
where the deterministic functions $B(t,x),\mu (t,x)$\ and $\sigma (t,x)$\
are continuous functions of their arguments. Moreover, for $(t,x,u)\in
\lbrack 0,T)\times 
\mathbb{R}
\times 
\mathbb{R}
$ given, we introduce the following running cost function in (\ref{p2}): 
\begin{equation}
f(t,x,u)=h(t,x)+k(t,x)u^{2}  \label{cost}
\end{equation}%
where $k(t,x)$ is a positive continuous function on $[0,T]\times 
\mathbb{R}
$, and the terminal cost $g(x)$ is such that $g(X_{T})\in L_{F_{T}}^{1}$.

Again, by a heuristic application of DPP as in (\ref{p8}), the value
function $V(s,x)$\ in equation (\ref{p3})\ satisfies the HJB equation 
\begin{eqnarray}
0 &=&\inf\limits_{u}\{f(t,x,u)+(\mathsf{L}^{u}V)(t,x)\}  \notag \\
&=&V_{t}+\frac{1}{2}\sigma ^{2}V_{xx}(t,x)+h(t,x)+\mu V_{x}(t,x)
\label{HjbE} \\
&&+\inf\limits_{u}\{k(t,x)u^{2}+B(t,x)uV_{x}\}.  \notag
\end{eqnarray}

Clearly, the minimum of $k(t,x)u^{2}+B(t,x)uV_{x}$ in (\ref{HjbE})\ is $%
\frac{-B^{2}(t,x)}{4k(t,x)}(V_{x})^{2}$\ with the\ minimizer 
\begin{equation}
u^{\ast }=\pi (t,x,V_{x}(t,x))=\frac{-B(t,x)}{2k(t,x)}V_{x}.  \label{7.5}
\end{equation}

Plugging (\ref{7.5}) into the equation (\ref{HjbE}), the HJB PDE takes the
form of 
\begin{eqnarray}
v_{t}+\frac{1}{2}\sigma ^{2}v_{xx}(t,x)+h(t,x)+\mu v_{x}(t,x)-\frac{%
B^{2}(t,x)}{4k(t,x)}(v_{x})^{2} &=&0  \label{7.7} \\
v(T,x) &=&g(x)  \notag
\end{eqnarray}%
which can also be written as 
\begin{eqnarray}
v_{t}+\frac{1}{2}\sigma ^{2}v_{xx}(t,x)+\mu v_{x}(t,x)+F(t,x,\sigma
v_{x}(t,x)) &=&0  \label{8} \\
v(T,x) &=&g(x)  \notag
\end{eqnarray}%
with 
\begin{equation}
F(t,x,z)=h(t,x)-\frac{H(t,x)}{2}z^{2},\text{ for }(t,x,z)\in \lbrack
0,T]\times 
\mathbb{R}
\times 
\mathbb{R}
,  \label{9}
\end{equation}%
where $H(t,x)=\frac{B^{2}(t,x)}{2k(t,x)\sigma ^{2}(t,x)}$.

\begin{remark}
(a) In a standard LQR problem, the state dynamics in (\ref{6}) has the usual
linear drift and diffusion coefficients, $k(t,x)=k(t)$ and $B(t,x)=B(t)$.\
Moreover, $h(t,x)$ and $g(x)$\ are quadratic cost functions. It is well
known that the condition $k(.)\geq 0$\ ($k(.)>0$, respectively) is a
necessary (sufficient, respectively)\ condition for the standard LQR problem
to be solvable.\newline
(b) If the value function solves the quasilinear PDE (\ref{7.7}), then the
PDE representation given by (\ref{8})-(\ref{9}) can be utilized for a BSDE
representation of the problem by Theorem \ref{t1} or Lemma \ref{l1}. \newline
\end{remark}

\subsection{2.2\quad The FBSDE Interpretation}

In view of Theorem \ref{t1} (and using the similar steps as in Lemma \ref{l1}%
), if the value function $V(t,x)$\ is a smooth solution to the equations (%
\ref{8})-(\ref{9}), then the pair $(Y_{t}^{s,x},Z_{t}^{s,x})$ with $%
Y_{t}=V(t,\tilde{X}_{t})$ and $Z_{t}=\sigma (t,\tilde{X}_{t})V_{x}(t,\tilde{X%
}_{t})$ is a solution to the BSDE 
\begin{eqnarray}
dY_{t}^{s,x} &=&-F(t,\tilde{X}_{t},Z_{t})dt+Z_{t}dW_{t}  \label{18} \\
Y_{T}^{s,x} &=&g(\tilde{X}_{T})  \notag
\end{eqnarray}%
with 
\begin{equation}
\tilde{X}_{t}^{s,x}=x+\int\limits_{s}^{t}\mu (t,\tilde{X}_{r}^{s,x})dr+\int%
\limits_{s}^{t}\sigma (r,\tilde{X}_{r}^{s,x})dW_{r},  \label{18b}
\end{equation}%
and$\ F(t,x,z)$ as in (\ref{9}). However, in general, we don't know if the
PDE has a classical solution since the function $F(s,x,z)$ doesn't satisfy
the usual Lipshitz or linear growth conditions in the state variables $x$
and $z$. It also gets more complicated when the terminal condition $g(X_{T})$
is not bounded.

\begin{remark}
\label{rep}(a) \textit{The representation of (\ref{18})-(\ref{18b}) as a
FBSDE system is not unique. Another representation (as in Lemma \ref{l1})
may be given by the following system, by eliminating the drift term of the
forward process:} 
\begin{eqnarray}
\hat{X}_{t}^{s,x} &=&x+\int\limits_{s}^{t}\sigma (r,\hat{X}_{r})dW_{r}
\label{19} \\
Y_{t}^{s,x} &=&g(\hat{X}_{T})+\int\limits_{t}^{T}\hat{F}(r,\hat{X}%
_{r},Z_{r})dr-\int\limits_{t}^{T}Z_{r}dW_{r}  \notag
\end{eqnarray}%
\textit{where the new driver function} \textit{is }$\hat{F}(t,x,z)=F(t,x,z)+%
\frac{\mu (t,x)}{\sigma (t,x)}z.$ \textit{Each representation has some
advantages depending on the complexity level of the forward and backward
equations in (\ref{18})-(\ref{19}). In this section, the representation (\ref%
{18}) will be used frequently based on the assumption that the forward state
dynamics (\ref{18b}) has a unique solution.}\newline
(b) \textit{When the diffusion term }$\sigma $\textit{\ is time dependent
only, the uniqueness of the solutions to (\ref{18}) can be shown even for
more general drift terms, as in Cetin (2012a). For the systems with
state-dependent diffusion terms }$\sigma (t,x)$ \textit{and nonlinear drift
terms, some more regularity or monotonicity assumptions would be needed.}%
\smallskip
\end{remark}

When the expressions $\sigma (t,x)$\ and $H(t,x)$\ are time-dependent only,
we have the following Theorem which is a special case of a result from Cetin
(2012) where the driver $F$ is also allowed to depend on $y$:

\begin{theorem}
\label{Main} Assume that \newline
(i) the SDE \textit{(\ref{18b}) has a unique solution }$\tilde{X}$\textit{\
in }$L_{F}^{2}[0,T]$\textit{\ with a.s. continuous paths \newline
(ii) the value function }$V(t,x)$\textit{\ in (\ref{p3}) is a smooth
solution of (\ref{8})-(\ref{9})\ \newline
(iii)} $\sigma (t,x)=\sigma (t),$ satisfying $\left\vert \sigma
(.)\right\vert >\delta >0$ uniformly on $[0,T]$. \textit{\newline
(iv) }the function $H(t,x)=H(t)$ is differentiable, and is such that $%
\left\vert \dot{H}(.)/H(.)\right\vert $\ is bounded on $[0,T]$. \textit{%
\newline
}Then the BSDE \textit{(\ref{18})} has a unique solution $(Y,Z)$ in $%
L_{F_{T}}^{2}\times L_{F}^{2}$.
\end{theorem}

\begin{remark}
Since the term $h(s,\tilde{X}(s))$\ is not bounded, the existence of a
solution is not guaranteed in general. However, if a function $v(t,x)$\
satisfies the HJB PDE (\ref{8})-(\ref{9}), then thanks to the monotonic
transformations $U_{t}=\exp (-H(t).Y)$ and $\Lambda _{t}=-H(t)U_{t}Z_{t}$,
the pair $[\exp (-H(t).v(t,\tilde{X}_{t}),-H(t)\sigma (t)v_{x}(t,\tilde{X}%
_{t})\exp (-H(t).v(t,\tilde{X}_{t}))]$ solves the BSDE with $0<U_{t}=\exp
(-H(t).v(t,\tilde{X}_{t})<1$.
\end{remark}

\begin{corollary}
\label{uniq}If the value function $V(t,x)$\ satisfies the HJB PDE (\ref{8})-(%
\ref{9}), then it is the unique solution of (\ref{8})-(\ref{9}).
\end{corollary}

\subsection{2.3 A Relevant LQR Problem with State-Independent Diffusion}

Now consider the linear state dynamics $\mu (t,x)=A(t)x$ and $\sigma
(t,x)=\sigma (t)$ in addition to the quadratic cost functions $%
f(t,x,u)=e^{-\lambda t}[(x-\xi (t))^{2}+k_{1}(t)u^{2}]$ and $%
g(x)=k_{2}(x-\xi (T))^{2}$, where all the time-dependent functions are
continuous. When $\xi (.)=0$, the problem reduces to the standard LQR
problem.\ By an appeal to the standard verification theorems for the
stochastic control problems (Fleming and Soner, 2006), the system (\ref{8})-(%
\ref{9}) can be shown to have a unique smooth solution $v(t,x),$ to the
corresponding LQR optimization problem in the following form:%
\begin{eqnarray}
V(t,x) &=&P(t)x^{2}+K(t)x+N(t)  \label{lqr} \\
V(T,x) &=&k_{2}(x-\xi (T))^{2},  \notag
\end{eqnarray}%
where, for $s\leq t<T$, the functions $P,K$ and $N$ solves the system of the
ODEs below:%
\begin{eqnarray}
\dot{P}(t)+e^{-\lambda t}+2AP(t)-e^{\lambda t}\frac{B^{2}}{k_{1}}P^{2}(t)
&=&0,\text{ }P(T)=k_{2}\geq 0,  \label{ric1} \\
\dot{K}(t)+(A(t)-e^{\lambda t}\frac{B^{2}}{k_{1}}P(t))K(t)-2e^{-\lambda
t}\xi (t) &=&0,\text{ }K(T)=-2e^{-\lambda T}\xi (T),  \label{odeK} \\
\dot{N}(t)+\sigma ^{2}P(t)-\frac{e^{\lambda t}}{4}\frac{B^{2}}{k_{1}}%
K^{2}(t)+e^{-\lambda t}\xi ^{2}(t) &=&0,\text{ }N(T)=e^{-\lambda T}\xi
^{2}(T).  \label{odeN}
\end{eqnarray}%
using the notation $\dot{f}(t)\equiv \frac{df(t)}{dt}$ for time derivatives.
It is well known (e.g., Fleming and Rishel, p.89) that the Riccati ODE (\ref%
{ric1}) has a non-negative (positive if $k_{2}>0$)\ solution$\ P(.)\in
C^{1}[0,T]$. Consequently, the linear first-order equations (\ref{odeK}) and
(\ref{odeN})$\ $also have unique solutions in $C^{1}[0,T]$. Moreover, an
optimal control process is given by 
\begin{equation*}
u_{t}^{\ast }=-e^{\lambda t}\frac{B}{2k_{1}}(t)[2P(t)X_{t}^{\ast }+K(t)]
\end{equation*}%
and the optimized state process which is given by the SDE%
\begin{eqnarray}
dX_{t}^{\ast } &=&(A(t)X_{t}^{\ast }+B(t)u(t,X_{t}^{\ast }))dt+\sigma
(t)dW_{t}, \\
&=&\{(A(t)-\frac{e^{\lambda t}}{2}\frac{B}{k_{1}}K(t)-e^{\lambda t}\frac{%
B^{2}P}{k_{1}}(t))X_{t}^{\ast }\}dt+\sigma (t)dW_{t}  \notag
\end{eqnarray}%
is a Gaussian process on $[0,T]$.

The case with $\sigma (t,x)=\sigma (t)x$ is also similar: The solution is a
quadratic function of $x$ with time dependent parameters being solutions to
ODEs similar to those above but we are not going to provide the details here
(note that the Theorem \ref{Main} doesn't apply directly in this case). In
both cases, the value function can be obtained by solving the corresponding
ODEs numerically. This can be done efficiently even in high dimensions so it
is not necessary to consider a FBSDE approach to solve such problems.%
\footnote{%
Actually, it is shown in Cetin (2006) that solving a corresponding FBSDE
system will result in a sequence of iterations which are equivalent to
solving the related ODEs above using Euler discretization.}\ However when
there is no explicit solution available, then the FBSDE approach could be
preferable, especially when the corresponding control problem or when the
PDE involves a state variable in higher dimensions. We describe such a
nonlinear application which is a generalized version of an example from Tsai
(1978).

\section{3.\ Nonlinear State Dynamics with State-Independent Diffusion}

Now consider a non-linear drift term while diffusion coefficient is still
depending only on time:%
\begin{eqnarray}
dX_{t} &=&[A(t)X_{t}+\delta r(t,X_{t})+B(t)u_{t}]dt+\sigma (t)dW(t),
\label{pert} \\
X_{0} &=&x_{0}>0  \notag
\end{eqnarray}%
where $\sigma (.)$ is bounded away from $0$ on the interval $[0,T]$, $p(t,x)$%
\ is a non-linear perburtation term, and $\delta $\ is the perturbation
constant. The\ system reduces to a linear one when $\delta =0$. Let the cost
functional for this perturbed problem be given 
\begin{equation}
J^{u}(s,x)=E_{s,x}\int\limits_{s}^{T}[l(t)(X_{t}-\xi
(t))^{2}+k(t)u_{t}^{2})]dt  \label{cost1}
\end{equation}%
where $k(.)>0$ and $\xi (.)$\ is continuous on $[0,T]$.\ Define the value
function as $V(s,x)=\inf_{u\in 
\mathbb{R}
}J^{u}(s,x)$, and let the control set $\mathcal{U}$\ consist of all square
integrable adapted processes $u_{t}$ such that the equation (\ref{pert}) has
a unique solution in $L^{2}([s,T],%
\mathbb{R}
)$ (however it is sufficient to consider only the feedback controls of the
Markovian form). This quadratic optimization problem cannot be solved
explicitly unless $\delta =0$ however assuming that the SDE (\ref{pert}) has
a solution for a sufficiently rich set of the control processes, including
the candidate optimal control $u^{\ast }=\frac{-B(t)}{2k(t)}V_{x}$\ (from (%
\ref{7.5})) and $u=0$, and the optimization problem (\ref{p2}) is solvable,
we can identify a corresponding FBSDE system to characterize the solution.
Since the terminal condition is bounded (in this case zero, for simplicity),
using the methods of the parabolic PDEs and stochastic analysis, it can be
shown to have a smooth solution (as in Tsai, 1978). For more general state
equations or terminal conditions, one can only expect to get a less smooth
(viscosity) solution, using either methods of the PDEs (as in Fleming and
Soner, 2006) or those of the FBSDEs. We know state some assumptions that
will be used for the main results of this section:

\begin{condition}
\label{C}Consider the equations (\ref{pert})-(\ref{cost1}) and let $\mu
(t,x)=A(t)x+\delta r(t,x)$, $\bar{\mu}(t,x,u)=\mu (t,x)+B(t)u$.\newline
(i) The time dependent functions $A,B,k,l$ and $\sigma $ are continuous on $%
[0,T]$.\newline
(ii) The function $l(t)$ is nonnegative, $k(t)$ is positive on $[0,T]$ and
is bounded away from zero: For some $\epsilon >0,$ $k(.)>\epsilon $.\newline
(iii) The function $H(t)=\frac{B^{2}(t)}{k(t)\sigma ^{2}(t)}$ is
continuously differentiable in $(0,T)$ and is bounded away from zero.\newline
(iv) (Monotonicity condition) For all $t\in \lbrack 0,T]$ and $x\in 
\mathbb{R}
$,\ $xr(t,x)\leq C(1+x^{2}),$ for some constant $C>0$.\newline
(v) (Monotonicity condition): For all $t\in \lbrack 0,T]$ and $x,y\in 
\mathbb{R}
$,\ $(x-y)r(t,x)-r(t,y)\leq K(x-y)^{2},$ for some constant $K>0$.
\end{condition}

\begin{lemma}
\label{state} Assume that the functions $A,B$ and $\sigma $ satisfy the
Condition \ref{C} (i), $r(t,x)$ is locally Lipshitz and satisfies the
Condition \ref{C} (iv) above. Then, \newline
(a) For every initial condition $\tilde{X}_{s}=x$, the control-free state
equation%
\begin{equation}
\tilde{X}_{t}^{s,x}=x+\int\limits_{s}^{t}[A(t)\tilde{X}_{v}^{s,x}+\delta r(t,%
\tilde{X}_{v}^{s,x})dv+\int\limits_{s}^{t}\sigma (v)dW_{v}  \label{unpert}
\end{equation}%
has a unique strong solution $\tilde{X}_{t}^{s,x}$\ in $L^{p}([s,T],%
\mathbb{R}
)$, for all $p\geq 2$. Moreover, with $\tilde{C}=\max \{\delta C+\frac{p-1}{2%
}\max\limits_{s\leq t\leq T}\sigma ^{2}(t),\delta C+\max\limits_{s\leq t\leq
T}A(t)\},$ it satisfies the following moment estimate for every $t\in
\lbrack s,T]$: 
\begin{equation*}
E_{s,x}\left\vert \tilde{X}_{t}^{s,x}\right\vert ^{p}\leq 2^{\frac{p}{2}%
-1}(1+\left\vert x\right\vert ^{p})e^{\tilde{C}p(t-s)}.
\end{equation*}%
\newline
(b) Let $u_{t}=u(t,X_{t})$ be a Markovian feedback control where $u(t,x)$\
is locally Lipshitz with respect to $x$ and also satisfies Condition \ref{C}
(iv), for some generic constant $C>0$. Then the SDE (\ref{pert}) has a
unique solution in $L^{p}([s,T],%
\mathbb{R}
)$, for all $p\geq 2,$ too, and a similar moment estimate as in part (a)
holds.\newline
(c) If both $r(t,x)$ and $u(t,x)$ satisfy a linear growth condition in $x$,
then the unique solution of (\ref{pert}) is in $S^{p}([s,T],%
\mathbb{R}
)$ with the estimate 
\begin{equation*}
E_{s,x}\sup\limits_{s\leq t\leq T}\left\vert X_{t}^{s,x}\right\vert ^{p}\leq
(1+\left\vert x\right\vert ^{p})e^{C_{1}(T-s)},
\end{equation*}%
where $C_{1}$ is a constant depending on $p,T,s$ and the linear growth
factor.
\end{lemma}

\begin{proof}
(a) By (iv), for all $t\in \lbrack s,T]$ and $x\in 
\mathbb{R}
$, $x\mu (t,x)+\frac{p-1}{2}\sigma ^{2}(t)\leq (A(t)+\delta C)x^{2}+\delta C+%
\frac{p-1}{2}\sigma ^{2}(t)\leq \tilde{C}(1+x^{2})$. Since $\mu (t,x)$ is
also locally Lipshitz, the result follows from the standard estimates based
on the Lyapunov function approach. See, for example, Mao (1997), Ch. 3,
Theorem 4.1. The proof of part (b) is similar, replacing $\mu (t,x)$ with $%
\bar{\mu}(t,x,u(t,x))$,\ and part (c) is a standard result for the SDEs with
coefficients of linear growth.\newline
\end{proof}

\begin{lemma}
\label{control}Let the assumptions of Condition \ref{C} hold and $u^{\ast }$
be an optimal control rule: $u_{t}^{\ast }=u(t,X_{t}^{^{\ast }})$ with $%
V(s,x)=J^{u^{\ast }}(s,x)$ and $X_{t}^{^{\ast }}=X_{t}^{u^{\ast }}$.
Moreover, let $X_{s}^{^{\ast }}=x=\tilde{X}_{s}.$\ Then we have the
following estimates:\newline
(a) There is a positive constant $C$ such that $0\leq V(s,x)\leq C(1+x^{2})$%
, for every $x\in 
\mathbb{R}
$, uniformly on $[0,T].$\newline
(b) $E_{s,x}\int\limits_{s}^{T}\left\vert u_{t}^{\ast }\right\vert
^{2}dt\leq \frac{C}{\epsilon }(1+\left\vert x\right\vert ^{2})$ where $C$ is
as in part (a)\newline
(c) $E_{s,x}(X_{t}^{\ast }-\tilde{X}_{t})^{2}\leq \tilde{C}(1+\left\vert
x\right\vert ^{2})$, for some positive constant $\tilde{C}$.

\begin{proof}
(a) For $u=0$, since $l(.)\geq 0$, we get $V(s,x)\leq J^{0}(s,x)\leq
E_{0,x}\int\limits_{0}^{T}[l(t)(\tilde{X}_{t}-\xi (t))^{2})]dt.$ Since $\xi
(.)$ and $l(.)$\ are bounded on $[0,T]$,\ the result follows from Lemma \ref%
{state} with $p=2$. \newline
(b) By (\ref{cost1}) and assumption (ii) of Condition \ref{C}, $%
V(s,x)=E_{s,x}\int\limits_{s}^{T}[l(t)(X_{t}^{\ast }-\xi
(t))^{2}+k(t)u_{t}^{\ast 2})]dt\geq $ $\epsilon
E_{s,x}\int\limits_{s}^{T}\left\vert u_{t}^{\ast }\right\vert ^{2}dt$. Hence
the inequality is a result of part (a).\newline
(c) First note that $d(X_{t}^{\ast }-\tilde{X}_{t})=\{A(t)(X_{t}^{\ast }-%
\tilde{X}_{t})+\delta (r(t,X_{t}^{\ast })-r(t,\tilde{X}_{t}))+B(t)u_{t}^{%
\ast }\}dt$ and let $\Delta (t)=X_{t}^{\ast }-\tilde{X}_{t},$\ for $s\leq
t\leq T$. Then by Ito's rule, 
\begin{equation*}
\Delta ^{2}(t)=\int\limits_{s}^{t}2A(v)\Delta
^{2}(v)dv+\int\limits_{s}^{t}2\delta (r(t,X_{t}^{\ast })-r(t,\tilde{X}%
_{t}))\Delta (v)dv+\int\limits_{s}^{t}2B(v)u_{v}^{\ast }\Delta (v)dv.
\end{equation*}%
By applying the monotonicity assumption (v) of Condition \ref{C} to the
second integral and the inequality $2ab\leq a^{2}+b^{2}$ to the third one
above, we get 
\begin{equation*}
0\leq \Delta ^{2}(t)\leq \int\limits_{s}^{t}\{(2A(v)+2\delta K+1)\Delta
^{2}(v)dv+\int\limits_{s}^{t}(B(v)u_{v}^{\ast })^{2}dv,
\end{equation*}%
where $K$ is as in Condition \ref{C} (v). Since $B(.)$ is bounded on $[0,T]$%
, by Lemma \ref{control}, $E_{s,x}\int\limits_{s}^{t}\left\vert
B(v)u_{v}^{\ast }\right\vert ^{2}dv\leq C_{1}(1+\left\vert x\right\vert
^{2}) $ for some positive constant $C_{1}$. Moreover, $\exists M>0$\ such
that $2A(t)+2\delta K+1\leq M$, for $0\leq t\leq T$. Therefore, by letting $%
g(s,x;t)=g(t)=E_{s,x}\Delta ^{2}(t)$, we obtain the inequality $0\leq
g(t)\leq C_{1}(1+\left\vert x\right\vert ^{2})+M\int\limits_{s}^{t}g(v)dv$.
By Gronwall's inequality, $g(t)\leq C_{1}(1+\left\vert x\right\vert
^{2})(1+M\int\limits_{s}^{t}e^{M(t-v)}dv)$ and hence the result follows.
\end{proof}
\end{lemma}

Now, by an appeal to Lemma \ref{state} part (a) with $p=2$, and part (c) of
the Lemma above, we get the following Corollary:

\begin{corollary}
(a) $E_{s,x}\left\vert X_{t}^{\ast }\right\vert ^{2}\leq K_{1}(1+\left\vert
x\right\vert ^{2})$, for some positive constant $K_{1}$.\newline
(b) $E_{s,x}\left\vert X_{t}^{\ast }\right\vert \leq K_{2}(1+\left\vert
x\right\vert )$, for some positive constant $K_{2}$.
\end{corollary}

\begin{theorem}
\label{Main1}Let the assumptions (i)-(v) hold\ and consider the perturbed
state dynamics (\ref{pert})\ together with the cost function (\ref{cost1})
and the value function $V(s,x)$. Then, \newline
(a) The value function $V(s,x)$ is in $C^{1,2}([0,T]\times 
\mathbb{R}
)$ such that

(i) it satisfies the HJB equation%
\begin{eqnarray}
v_{t}(t,x)+\frac{1}{2}\sigma ^{2}(t)v_{xx}(t,x)+F(t,x,\sigma
(t)v_{x})+\delta r(t,x)v_{x}(t,x) &=&0  \label{Pde} \\
v(T,x) &=&0,  \notag
\end{eqnarray}%
\ \ where $F(t,x,z)=(x-\xi (t))^{2}-\frac{H(t)z^{2}}{2}$ over $[0,T]\times 
\mathbb{R}
$.

(ii) $\exists C>0$ such that $\left\vert V_{x}(s,x)\right\vert \leq
C(1+\left\vert x\right\vert ),$ uniformly for $s\in \lbrack 0,T]$.\newline
(b) The process $u_{t}^{\ast }=u(t,X_{t})=\frac{-B(t)}{2k(t)}V_{x}(t,X_{t})$%
\ is an admissible feedback control rule such that $u(s,x)$ is the unique
minimizer of $J^{u}(s,x)$, for all $(s,x,u)\in \lbrack 0,T]\times 
\mathbb{R}
\times 
\mathbb{R}
$. \ Moreover, $u(.,x)$ is a locally Lipshitz function of $x$, and satisfies
a linear growth condition in $x$ (as in part (a)(ii) above).
\end{theorem}

\begin{proof}
(a) The proof relies on the approximation of the domain $[0,T]\times 
\mathbb{R}
$ by the compact subsets, and the standard but lenghty localization
arguments and passages to the limit for the Cauchy problems of second order
parabolic equations, as in Tsai (1978) and Fleming and Rishel (1975). We
skip these technical details since the steps involved are very similar to
those of Lemma 3.1 and Lemma 3.2 of Tsai (1978).\newline
(b) From the more general equations (\ref{HjbE}) and (\ref{7.5}), it is easy
to see that $u(s,x)=\frac{-B(t)}{2k(t)}V_{x}(t,x)$ is the unique minimizer
of $J^{u}(s,x)$ in (\ref{cost1}). It has a linear growth in $x$ by part
(a)(ii), and is square integrable by Lemma \ref{control}. It is locally
Lipshitz in $x$ uniformly in $t$ since $\frac{B(t)}{2k(t)}$ is bounded on $%
[0,T]$, and $V_{x}(t,x)$ is differentiable with respect to $x$.

\begin{remark}
Since no specific growth condition is assumed on the perturbation term $%
r(t,x)$, except the monotonicty and local Lipshitz properties, the standard
verification theorems and dynamic programming principle of stochastic
control theory may not directly apply. Again, some localization techniques
would help to ensure the uniqueness of the solutions to the PDE (\ref{Pde}).%
\footnote{%
Tsai (1978) take advantage of the differentiability (and the existence of an
upper bound on the derivative) that doesn't apply here.} However we follow a
BSDE approach to prove the uniqueness. We then provide a probabilistic
representation of both the value function and the optimal control.
\end{remark}
\end{proof}

\begin{theorem}
\label{Main2}Let $\tilde{X}_{t}^{s,x}$, $X_{t}^{\ast },$\ $u_{t}^{\ast }$
and $V(t,x)$\ be as before, and introduce the pair $%
(Y_{t}^{s,x},Z_{t}^{s,x}) $ as $Y_{t}=V(t,\tilde{X}_{t})$ and $Z_{t}=\sigma
(t)V_{x}(t,\tilde{X}_{t})$. Then,

(i) The pair $(Y_{t}^{s,x},Z_{t}^{s,x})$ is the unique solution to the BSDE 
\begin{equation}
dY_{t}^{s,x}=-F(t,\tilde{X}_{t},Z_{t})dt+Z_{t}dW_{t}\text{, }Y_{T}^{s,x}=0.
\end{equation}

Moreover, $V(t,x)=Y_{t}^{t,x}$ (the value of $Y_{t}$ when $X_{t}=x$)$,$ and $%
u(t,x)=\frac{-B}{2k\sigma }(t)Z_{t}^{t,x}$\ for $x>0$ and $t\in \lbrack 0,T)$%
. An optimal (feedback) control for the problem is given .

(ii) The value function is the unique classical solution of the PDE (\ref%
{Pde}) with a quadratic growth in $x$.\ Moreover, the control process $%
u_{t}^{\ast }$ is the unique optimal control for the optimization problem (%
\ref{pert})-(\ref{cost1}).
\end{theorem}

\begin{proof}
The uniqueness follows from the Theorem \ref{Main} and the Corollary \ref%
{uniq}. Then the optimality is a direct result of Part (ii) is a result of
Theorem \ref{Main1} and the uniqueness. The BSDE representation in (i) is
obtained similar to that of (\ref{18})-(\ref{18b})\ and Remark \ref{rep} by
following the same steps as in \textit{Lemma \ref{l1}. }
\end{proof}

\begin{corollary}
The triple $(\tilde{X}_{t}^{s,x},Y_{t}^{s,x},Z_{t}^{s,x})$ satisfies\newline
(a) $\sup\limits_{s\leq t\leq T}E_{s,x}\left\vert Y_{t}\right\vert
+E_{s,x}[\int\limits_{s}^{T}Z_{t}^{2}dt+\int\limits_{s}^{T}\tilde{X}%
_{t}^{2}dt]\leq C(1+x^{2})$.\newline
(b) If the perturbation term $r(t,x)$ has a linear growth, then $%
E_{s,x}\sup\limits_{s\leq t\leq T}\left\vert Y_{t}\right\vert
+E_{s,x}[\int\limits_{s}^{T}Z_{t}^{2}dt+\int\limits_{s}^{T}\tilde{X}%
_{t}^{2}dt]\leq C(1+x^{2})$\ also holds.
\end{corollary}

\begin{proof}
It is a result of Lemma \ref{control}, Theorem \ref{Main1} part (a) (ii),
and Theorem \ref{Main2}, by utilizing Lemma \ref{state} (a) for the proof of
part (a), and Lemma \ref{state} (c) for the proof of part (b), with $p=2$.
\end{proof}

The following example is a slightly generalization of an application from
Tsai (1978). The well-posedness of more general cases are discussed in Cetin
(2012a).

\begin{example}
Consider the following controlled state dynamics with constant parameters: 
\begin{eqnarray*}
dX_{t} &=&(-\delta X_{t}^{3}+Bu_{t})dt-\sigma dW_{t} \\
X_{0} &=&x_{0}>0.
\end{eqnarray*}%
Let the cost functional for this perturbed problem be given by $J^{\delta
,u}(s,x)=E_{s,x}\int\limits_{s}^{1}[(X_{t}-\xi )^{2}+ku_{t}^{2})]dt$ where $%
k>0$ and the value function is $V^{\delta }(s,x)=\inf_{u}J^{\delta ,u}(s,x)$%
. By the Theorem above, $V^{\delta }(s,x)$\ is the unique solution to the
PDE 
\begin{eqnarray*}
v_{t}(t,x)+\frac{1}{2}\sigma ^{2}v_{xx}(t,x)+(x-\xi )^{2}-\delta
x^{3}v_{x}(t,x)-Cv_{x}^{2}(t,x)/4 &=&0 \\
v(T,x) &=&0
\end{eqnarray*}%
\ where $C=B^{2}/k>0$ over $[0,T]$.\ Again, it is not likely to obtain an
explicit solution of this nonlinear and one needs to follow a numerical
procedure to solve the equation and hence describe the behavior of the
optimal action (control). Note that for the unperturbed LQR problem ($\delta
=0$) with the constant parameters $C$ and $\xi $,\ it can be verified (using
the results in subsection 2.3) that the solution is given by the quadratic
expression $V^{0}(s,x)=\lambda (t)(x-\xi )^{2}+\gamma (t)$, where $\lambda
(t)=\frac{\tanh (\sqrt{C}(1-t))}{\sqrt{C}}$ and $\gamma (t)=\frac{\sigma ^{2}%
}{C}\ln \cosh (\sqrt{C}(1-t))$. Moreover, the optimal control $u_{t}^{\ast
,0}$ is a linear feedback control: $u_{t}^{\ast ,0}=-\frac{B}{2k}%
(t)V_{x}^{0}(t,x)=\frac{-sign(B)}{\sqrt{k}}\tanh (\frac{\left\vert
B\right\vert }{\sqrt{k}}(1-t))(x-\xi )$. When the state equation deviates
from a linear dynamics significantly, then it may be important to know how
the corresponding optimal action and the value function differ from the
unperturbed LQR setup. The discussion of this problem and its numerical
solution using a probabilistic appraoch are considered in Cetin (2012b). One
can also refer to Tsai (1978) and Nishikawa et al. (1976) for a PDE approach
for the details. Some results are given below without proof:

\begin{itemize}
\item The value function $V^{\delta }(s,x)$ can be approximated as follows:%
\begin{equation*}
V^{\delta }(s,x)=V^{0}(s,x)+2\delta \lbrack
K_{1}(s)x^{4}+K_{2}(s)x^{2}]+O(\delta ^{2})
\end{equation*}%
where $V^{0}(s,x)$ is\ the value function for the unperturbed problem and $%
K_{1}(.),K_{2}(.)\in C^{1}[0,1].$

\item The optimal control is%
\begin{equation*}
u^{\ast ,\delta }(s,x)=u^{\ast ,0}(s,x)-2\delta \lbrack
4K_{1}(s)x^{3}+2K_{2}(s)x^{3}]+O(\delta ^{2})
\end{equation*}%
where $u^{\ast ,0}(s,x)$ is as above.

\item $u^{\ast ,\delta }(s,x)\rightarrow u^{\ast ,0}(s,x)$, as $\delta
\rightarrow 0$ uniformly on $[0,T]\times Q$, for any compact set $Q$\ of $%
\mathbb{R}
.$
\end{itemize}
\end{example}

\renewcommand{\baselinestretch}{1}\small\normalsize%


\begin{thebibliography}{99}
\bibitem{cetin} Cetin, C. (2005). Backward stochastic differential equations
with quadratic growth and their applications. Ph.D. Dissertation, USC.

\bibitem{Cetin1} Cetin, C. (2012a). Uniqueness of solutions to certain
Markovian backward stochastic differential equations, arXiv: 1210-8230

\bibitem{Cetin2} Cetin, C. (2012b). A forward-backward numerical scheme for
the solution of perturbed linear-quadratic control problems, working paper

\bibitem{EPQ} El Karoui, N., Peng, S., Quenez, M. C. (1997). Backward
stochastic differential equations in finance. \textit{Math. Finance} \textbf{%
7} 1-71

\bibitem{FR} Fleming, W. H., and Rishel, R. W. (1975). Deterministic and
Stochastic Optimal Control, Springer-Verlag

\bibitem{FS} Fleming, W.H.\ and Soner, H.M. (2006). Controlled Markov
Processes and Viscosity Solutions, Springer-Verlag, Second Edition.

\bibitem{K} Kobylanski, M (2000). Backward stochastic differential equations
and partial differential equations with quadratic growth. \textit{The Annals
of Prob}. \textbf{28} 558-602

\bibitem{LSM} Lepeltier, J.P. and San Martin, J. (1997). Backward stochastic
differential equations with continuous coefficients. \textit{Statist Probab.
Lett.} \textbf{32} 425-430

\bibitem{LSM98} Lepeltier, J.P. and San Martin, J. (1998). Existence for
BSDE with superlinear-quadratic coefficient. \textit{Stochastics Stochastic
Rep.} \textbf{63} 227-240

\bibitem{MPY} Ma, J., Protter, P. and Yong, J. (1994). Solving
forward-backward stochastic differential equations explicitly-a four step
scheme. \textit{Prob. Theory Related Fields} \textbf{98}, 339-359

\bibitem{MaYong} Ma, J., and Yong, J. (1999). Forward-Backward Stochastic
Differential Equations and Their Applications, Lecture Notes in Math, 1702,
Springer

\bibitem{Mao} Mao, X. (1995). Adapted solution of backward stochastic
differential equations with non-Lipshitz coefficients. \textit{Stoch.
Process. Appl.} \textbf{58} 281-292

\bibitem{Mao97} Mao, X. (1997). Stochastic Differential Equations and Their
Applications, Horwood Publishing.

\bibitem{Nish} Nishikawa, Y., N. Sannomiya and H. Itakura (1976). \textit{%
Power series approximations for the optimal feedback control of noisy
nonlinear systems}, J. Optim Theory and Appl., Vol 19, No. 4.

\bibitem{PP1990} Pardoux, E. and Peng, S. (1990). Adapted solution of a
backward stochastic differential equation. \textit{Systems Control Lett.} 
\textbf{14} 55-61

\bibitem{PP1992} Pardoux, E. and Peng, S. (1992). Backward stochastic
differential equations and quasilinear partial differential equations,
Lecture Notes in CIS, vol. \textbf{176} Springer, 200-217

\bibitem{PP1994} Pardoux, E. and Peng, S. (1994). Some backward stochastic
differential equations with non-Lipshitz coefficients, Pr\'{e}publication
LATP, 94-03

\bibitem{Peng92} Peng, S. (1992). Stochastic Hamilton-Bellman equations, 
\textit{SIAM J. Control Optim}., \textbf{30}, 284-304

\bibitem{Tsai78} Tsai, C-P. (1978). Perturbed stochastic linear regulator
problems, \textit{SIAM J. Control Optim}., \textbf{16}, 396-410.

\bibitem{YZ} Yong, J.\ and Zhou, X.Y. (1999). Stochastic Controls:
Hamiltonian Systems and HJB Equations, Springer-Verlag, New York
\end{thebibliography}
\end{document}